\documentclass{article}

\usepackage[english]{babel}
\usepackage[utf8]{inputenc}
\usepackage{amsmath}
\usepackage{amssymb}
\usepackage{amsthm}
\usepackage{graphicx}
\usepackage{mathtools}
\usepackage[T1]{fontenc}
\usepackage{amsthm}
\usepackage{setspace}
\usepackage{amsfonts}
\usepackage{eucal}
\usepackage{mathrsfs}
\usepackage{pictexwd,dcpic}
\usepackage{tikz}
\usepackage{tqft}
\usepackage{pigpen}
\usepackage[colorinlistoftodos]{todonotes} 
\usepackage[colorlinks=true,
		     allcolors=black]{hyperref}              
             
\usepackage{fancyhdr} 
\usepackage{sectsty}

\newtheorem{lemma}{Lemma}
\newtheorem{theorem}{Theorem}

\newtheorem{corollary}{Corollary}
\newtheorem{definition}{Definition}
\newtheorem{remark}{Remark}
\newtheorem{question}{Question}

\title{\Large \scshape \textbf{Proper base change over henselian pairs}}
\author{\large \scshape Massimo Pippi}
\date{}
\allsectionsfont{\mdseries\scshape\normalsize\centering}
\begin{document}
\maketitle
\begin{center}
{\footnotesize \scshape Université Paul Sabatier}\\
{\footnotesize \scshape UMR 5219 du CNRS}\\
{\footnotesize \scshape Institut de Mathématiques de Toulouse}\\
{\footnotesize \scshape 118 route de Narbonne}\\
{\footnotesize \scshape F-31062 TOULOUSE Cedex 9}\\
{\footnotesize \itshape E-mail address:} {\footnotesize Massimo.Pippi@math.univ-toulouse.fr}\\
{\footnotesize \scshape ORCID ID: 0000-0002-5660-3156}\\
{\scriptsize The author is supported by the NEDAG PhD grant ERC-2016-ADG-741501}\footnote{the content of this note was obtained before the author started to receive the above mentioned scholarship, but it was organized in the present manuscript afterwards.}
\begin{minipage}{0.9\textwidth}
\vspace{0.5cm}
{\footnotesize{\scshape Abstract.}
We discuss a question which appears in [\textit{S\'eminaire de G\'eom\'etrie Alg\'ebrique du Bous Marie- Th\'eorie des topos et cohomologie \'etale des sch\'emas, Exposé XII, Remarks 6.13}] concerning proper base change. In particular, we propose a solution in a particular non-affine case.}
\end{minipage}
\end{center}
\section{Introduction}
The question we would like to answer is the following one\footnote{for the notation and the exact definition of the objects that are involved we refer to the original source.}: 
\begin{question}\label{conj1}
Let $(X,X_0)$ be an henselian couple (in the sense of \textnormal{Definition \ref{definition henselian couple}}). Is it true that \it{(ii)} and \it{(iii)} in \textnormal{\cite[Exposé XII, Proposition 6.5]{SGA4}} hold with $\mathbb{L}=\mathbb{P}$ and for every $n$?
\end{question}
This question appears in \cite[Exposé XII, Remarks 6.13]{SGA4}.  

 
We can restate it as follows:
\begin{question}\label{conj2}
Let $(X,X_0)$ be an henselian couple. Is it true that
\begin{enumerate}
\item[a.] the base change functor induces an equivalence between the category of étale coverings of $X$ and the category of étale coverings of $X_0$?
\item[b.]for any torsion étale sheaf $\mathscr{F}$ and for any integer $n$, the morphism
\[
H^n(X,\mathscr{F})\longrightarrow H^n(X_0,\mathscr{F}_{|X_0})
\]
is an isomorphism?
\end{enumerate}
\end{question}

\begin{remark}
\begin{enumerate}
\item When $X$ is proper and finitely presented over an henselian ring $(A,m)$ and $X_0=X\times_{Spec(A)} Spec(A/m)$, we know that the answer to \textnormal{Question \ref{conj1}} is affirmative. This is the proper base change theorem in étale cohomology.
\item The case $(X,X_0)=(Spec(A),Spec(A/I))$\footnote{an affine henselian couple is the same thing as an henselian pair. See Remark \ref{henselian pairs are affine henselian couples}.}, for $(A,I)$ an henselian pair was studied an solved by R. Elkik in \textnormal{\cite{Elk}} and by O. Gabber in \textnormal{\cite{Gab}}.
\end{enumerate}
\end{remark}
We propose a solution in the following situation:\\ 

$(\dagger)$ \hspace{0.5cm}{\itshape Let $X$ be proper over a noetherian affine scheme $Spec(A)$ and $X_0=X\times_{Spec(A)} Spec(A/I)$ for some ideal $I\subseteq A$.}\\
 
We will see that, under these assumptions, $(X,X_0)$ is an henselian couple for which Question \ref{conj1} has a positive answer. To achieve this, we will first generalize \cite[Theorem 3.1]{Art69} to the following form: 
\begin{theorem}\label{proper base change over henselian pairs}
Let $(A,I)$ be an henselian pair. Let $S=Spec(A)$ and let $f:X\longrightarrow S$ be a proper finitely presented morphism. Let $X_0=X\times_S S_0$, where $S_0=Spec(A/I)$. Then
\[
\text{É}t_f(X)\longrightarrow \text{É}t_f(X_0)
\]
\[
Z\mapsto Z\times_S S_0
\]
is an equivalence of categories.
\end{theorem}
Here \textit{É}$t_f(W)$ denotes the category of finite étale schemes over $W$. The key tools for the proof are Artin's approximation theory and \cite[Tag 0AH5]{SP}, which combined with \cite[Corollary 1.8]{Art69} yields the following theorem
\begin{theorem}\label{second answer to Artin's first question}
Let $(A,I)$ be an henselian pair with $A$ noetherian. Let $\hat{A}$ be the $I$-adic completion of $A$ and assume that one of the following hypothesis is satisfied:
\begin{enumerate}
\item $A\longrightarrow \hat{A}$ is a regular ring map;
\item $A$ is a G-ring;
\item $(A,I)$ is the henselization\footnote{here the henselization is the left adjoint to the inclusion functor \textit{Henselian Pairs}$\longrightarrow$ \textit{Pairs}} of a pair $(B,J)$, where $B$ is a noetherian G-ring.
\end{enumerate} 
Let $\mathscr{F}$ be a functor which is locally of finite presentation\footnote{see \cite[Definition 1.5]{Art69} } 
\[
\textit{A-algebras}\longrightarrow \textit{Sets}
\]
Given any $\hat{\xi}\in \mathscr{F}(\hat{A})$ and any $N\in \mathbb{N}$, there exists an element $\xi \in \mathscr{F}(A)$ such that
\[
\xi \equiv \hat{\xi} \text{ } mod \text{ }I^N
\]
i.e. $\xi$ and $\hat{\xi}$ have the same image in $\mathscr{F}(A/I^N)\cong \mathscr{F}(\hat{A}/\hat{I}^N)$
\end{theorem}
\begin{remark}
In \textnormal{Theorem \ref{second answer to Artin's first question}} we have that $\textit{3.} \Rightarrow \textit{2.} \Rightarrow \textit{1.}$ See \textnormal{\cite[Tag 0AH5]{SP}}.
\end{remark}
\begin{remark}\label{first step in the solution}
\textnormal{Theorem \ref{proper base change over henselian pairs}}, joined with \textnormal{\cite[Corollary 1]{Gab}}, gives us a positive answer to \textnormal{Question \ref{conj1}} when $(X,X_0)$ is proper and finitely presented over an henselian pair. Moreover, we will see how we can always reduce to this case from situation $(\dagger)$.
\end{remark}

\section{Proof of Theorem \ref{proper base change over henselian pairs}}
This proof is an adaption of the one given in the local case by Artin (see \cite[Theorem 3.1]{Art69}). This generalization is possible thanks to Popescu's characterization of regular morphisms between noetherian rings, which provides us Theorem \ref{second answer to Artin's first question} as a corollary.\\
First we reduce to the case where $A$ is the henselization of a  finitely presented $\mathbb{Z}$-algebra. in order to do this, we need the following two preliminary lemmas.
\begin{lemma}\label{finite étale covers modulo iso functor locally of finite presentation}
Let $S=Spec(A)$ and let $g:X \longrightarrow S$ be a proper morphism of finite presentation. Then the functor
\[
\mathscr{F} : \textit{A-Algebras}\longrightarrow \textit{Sets}
\]
\[
B \mapsto \{\text{finite étale coverings of } Spec(B)\times_S X\}/\text{isomorphism}
\]
is locally of finite presentation.
\end{lemma}
\begin{proof}
See the beginning of the proof of \textnormal{\cite[Theorem 3.1]{Art69}}.
\end{proof}

\begin{lemma}\label{hom locally of finite presentation}
Let $S=Spec(A)$ and let $g:X \longrightarrow S$ be a proper morphism of finite presentation. Let $Z_1\rightarrow X$ and $Z_2\rightarrow X$ be two finite étale covers of $X$. Then the functor
\[
\mathscr{G}: \textit{A-algebras}\longrightarrow \textit{Sets}
\]
\[
B\mapsto Hom_{X\times_S Spec(B)}(Z_1\times_S Spec(B),Z_2\times_S Spec(B))
\]
is locally of finite presentation.
\end{lemma}
\begin{proof}
The lemma is a straightforward consequence of \cite[Theorem 8.8.2.(i)]{EGA4.3}.
\end{proof}
Let $(A,I)$ be an henselian pair and write $A$ as a direct limit $\varinjlim A_i$, where each $A_i$ is a subalgebra of $A$ that is finitely generated over $\mathbb{Z}$. Let $(A_i^h,(I\cap A_i)^h)$ be the henselization of $(A_i,(I\cap A_i))$ for each $i$. Then by \cite[Chapter XI, Proposition 2]{Ray} $\varinjlim (A_i^h,(I\cap A_i)^h)$ is an henselian pair. It is easy to see that 
\[
(A,I)=\varinjlim (A_i^h,(I\cap A_i)^h)
\]
Write $S_i=Spec(A_i^h)$ for every index $i$. Then
\[
S=\varprojlim S_i
\]
By \cite[Thereom 8.8.2. (ii)]{EGA4.3} we know that $X$ comes from a finitely presented scheme $X_{i_0}$ for some index $i_0$, i.e. $X\cong X_{i_0}\times_{S_{i_0}}S$. Moreover, by \cite[Theorem 8.10.5]{EGA4.3}, we can assume that $X_{i_0}$ is also proper over $S_{i_0}$. As the functor
\[
\mathscr{F}: A^h_{i_0}-Algebras\longrightarrow Sets
\]
\[
B \mapsto \{\text{finite étale coverings of } Spec(B)\times_{S_{i_0}} X_{i_0}\}/\text{isomorphism}
\]
is locally of finite presentation, we have that
\[
\mathscr{F}(A)=\varinjlim \mathscr{F}(A^h_i)
\]
Therefore, every finite étale cover of $X$ comes from a finite étale cover of $X_i=S_i\times_{S_{i_0}}X_{i_0}$
for a suitable index $i$.
\begin{remark}
All schemes $X_{i_0}\times_{S_{i_0}}S_i$  and $X\cong X_{i_0}\times_{S_{i_0}}S$ are quasi-compact and quasi-separated, as they are proper over affine schemes. 
\end{remark}
Let $Z\rightarrow X$ and $W\rightarrow X$ be two finite étale covers of $X$. Then we can assume without loss of generality that they come from two finite étale covers $Z_{i_0}\rightarrow X_{i_0}$, $W_{i_0}\rightarrow X_{i_0}$. Then by Lemma \ref{hom locally of finite presentation} we see that
\[
\varinjlim Hom_{X_i}(Z_i,W_i)=Hom_X(Z,W)
\]
It is then clear that we can reduce the proof of Theorem \ref{proper base change over henselian pairs} to the case where $(A,I)$ is the henselization of a pair $(B,J)$, where $B$ is finitely generated over $\mathbb{Z}$. In particular, $B$ is a G-ring and Theorem \ref{second answer to Artin's first question} holds.

\begin{lemma}\label{essential surjectivity}
The functor in Theorem \ref{proper base change over henselian pairs} is essentially surjective.
\end{lemma}
\begin{proof}
Consider a finite étale morphism $X_0'\longrightarrow X_0$. Label $\hat{A}$ the completion of $A$ with respect to the ideal $I$ and let $\hat{S}=Spec(\hat{A})$, $\hat{X}=X\times_S \hat{S}$. Notice that $\hat{A}$ is a complete separated ring by Krull's theorem (see \cite[Theorem 10.17]{AM}). By \cite[Theorem 18.3.4]{EGA4.4}, we have that the functor
\[
\textit{É}t_f(\hat{X})\longrightarrow \textit{É}t_f(X_0)
\]
\[
Z\mapsto Z\times_S S_0
\]
is an equivalence of categories. Then there exists some $[\hat{X}'\longrightarrow \hat{X}]\in \mathscr{F}(\hat{A})$ such that
\[
\hat{X}'\times_{\hat{S}}S_0\cong X_0'
\]
By Theorem \ref{second answer to Artin's first question} we get that there exists some finite étale morphism $X'\longrightarrow X$ which is congruent modulo $I$ to $\hat{X}'\longrightarrow \hat{X}$, i.e.
\[
X'\times_S S_0\cong X_0'
\]
\end{proof}
It remains only to show that the functor in Theorem \ref{proper base change over henselian pairs} is fully faithful. 
\begin{lemma}\label{fully faithfulness}
The functor in \textnormal{Theorem \ref{proper base change over henselian pairs}} is fully faithful.
\end{lemma}
\begin{proof}
Let $X'$ and $X''$ be two finite étale schemes over $X$ and let $\phi \in Hom_X(X',X'')$. The morphism $\phi$ corresponds uniquely to its graph $\Gamma_{\phi} : X'\longrightarrow X'\times_X X''$, which is an open immersion as both $X'$ and $X''$ are of finite type over $X$ and as $X''$ is étale over $X$ (see \cite[Corollaire 3.4]{SGA1}). Also notice that $\Gamma_{\phi}$ is a closed immersion (see \cite[Exercise 3.3.10]{L}). If we assume that $X'$ is connected and nonempty, $\phi$ corresponds uniquely to a connected component of $X'\times_X X''$ of degree one over $X'$. The degree of such a component can be measured at any point of $X'$. We conclude therefore by applying the next lemma to a component of $X'\times_X X''$.
\end{proof}
\begin{lemma}
$X$ is nonempty and connected if and only if the same is true for $X_0$. 
\end{lemma}
\begin{proof}
We are given the following cartesian square
\[
\begindc{\commdiag}[20]
\obj(0,30)[1]{$X_0$}
\obj(30,30)[2]{$X$}
\obj(0,0)[3]{$S_0$}
\obj(30,0)[4]{$S$}
\mor{1}{2}{$$}
\mor{1}{3}{$$}
\mor{2}{4}{$f$}
\mor{3}{4}{$$}
\enddc
\]
If $X$ is connected and nonempty, then $f(X)\subseteq S$ is a nonempty closed subset of $S$ (as $f$ is proper). Let $J$ be an ideal of $A$ that identifies $f(X)$. Let $f(x)=p \in V(J)$ be a closed point of $S$. As $I$ is contained in the Jacobson radical of $A$, the prime ideal $p$ lies in $S_0$. Then
\[
\begindc{\commdiag}[20]
\obj(0,30)[1]{$X_0$}
\obj(30,30)[2]{$X$}
\obj(0,0)[3]{$S_0$}
\obj(30,0)[4]{$S$}
\obj(-20,50)[5]{$\{x\}$}
\mor{1}{2}{$$}
\mor{1}{3}{$$}
\mor{2}{4}{$f$}
\mor{3}{4}{$$}
\mor{5}{2}{$$}
\mor{5}{3}{$$}
\mor{5}{1}{$$}[-1,1]
\enddc
\]
In particuar, $X_0$ is nonempty. Furthermore, as this argument can be used for any connected component of $X$, if $X$ is disconnected then also $X_0$ is disconnected.\newline
Conversely, assume that $X_0$ is disconnected. Label $C_0$ a nonempty connected component of $X_0$. As the scheme $X_0$ is quasi-compact, $C_0$ is open and closed in $X_0$. Therefore, $C_0\longrightarrow X_0$ is a finite étale morphism. By Lemma \ref{essential surjectivity}, there exists a finite étale morphism $C\longrightarrow X$ which induces $C_0\longrightarrow X_0$. As $C_0$ is connected and nonempty, the same is true for $C$. The morphism $C\longrightarrow X$ is therefore of degree $1$ at every point of $C$. As it is also finite and étale, it is both an open and a closed immersion, i.e. $C$ is a connected component of $X$. If $C=X$, we would get $C_0=X_0$, a contradiction. Then $X$ is disconnected. Finally, it is clear that if $X_0$ is nonempty, $X$ is nonempty too. 
\end{proof}
Theorem \ref{proper base change over henselian pairs} follows immediately from Lemma \ref{essential surjectivity} and Lemma \ref{fully faithfulness}.

\section{Henselian couples}
Recall that an henselian pair $(A,I)$ is a ring $A$ together with an ideal $I\subseteq$ such that 
\begin{enumerate}
\item $I$ is contained in the Jacobson ideal of $A$;
\item for every finite $A$ algebra $B$, there is a bijection between the set of idempotent elements of $B$ and the set of idempotent elements of $B\otimes_A A/I$.
\end{enumerate} 
For more details, see \cite{Ray}.\\
Let $(A,I)$ be an henselian pair. Then for every finite morphism $Spec(B)=X\longrightarrow Spec(A)$, we have a bijection
\[
Id(B)=Of(X)=Of(X_0)=Id(B/IB) \hspace{0.5cm} \text{where } X_0=X\times_{Spec(A)} Spec(A/I)
\]
Here $Of(Z)$ denotes the set of subsets of $Z$ which are both open and closed.\\
This fact suggests the following definition (see \cite[Définition 18.5.5]{EGA4.4}), which is meant to generalize the notion of henselian pair to the non-affine setting.
\begin{definition}\label{definition henselian couple}
Let $X$ be a scheme and let $X_0$ be a closed subscheme. We say that $(X,X_0)$ is an \textit{henselian couple} if for every finite morphism $Y\longrightarrow X$ we have a bijection
\[
Of(Y)=Of(Y_0)
\]
where $Y_0=Y\times_X X_0$.
\end{definition}
\begin{remark}
If $X$ is locally noetherian, it is a consequence of \textnormal{\cite[Proposition 6.1.4]{EGA1}}  and \textnormal{\cite[Corollaire 6.1.9]{EGA1}} that connected sets in $Of(X)$ (resp. $Of(X_0)$) are in bijection with $\Pi_0(X)$ (resp. $\Pi_0(X_0)$), the set of connected components of $X$ (resp. $X_0$).
\end{remark}
\begin{remark}
It is a consequence of \textnormal{\cite[Corollary 5.1.8]{EGA1}} that $(X,X_0)$ is an henselian couple if and only if $(X_{red},(X_0)_{red})$ is an henselian couple as well.
\end{remark}
\begin{remark}\label{henselian pairs are affine henselian couples}
It is immediate to observe that if $(A,I)$ is a pair and $(Spec(A),Spec(A/I))$ is an henselian couple, then $I$ is contained in the Jacobson radical of $A$. In fact, if $m\subseteq A$ is a maximal ideal, then we have a bijection
\[
Of(Spec(A/m))=Of(Spec(A/m\otimes_A A/I))
\]
In particular, $Spec(A/m\otimes_A A/I)$ can not be the empty scheme. Therefore, as it is a closed subscheme of $Spec(A/m)$, we must have an equality $Spec(A/m)=Spec(A/m\otimes_A A/I)$, whence $I\subseteq m$. Moreover, if $Z\longrightarrow Spec(A)$ is a finite morphism, then $Z=Spec(B)$ is affine and the corresponding morphism $A\longrightarrow B$ is finite. Then we have bijections
\[
Id(B)=Of(Spec(B))=Of(Spec(B/IB))=Id(B/IB)
\]
We have just showed that an affine henselian couple is an henselian pair. The converse was observed at the beginning of this section. 
\end{remark}
\begin{lemma}\label{a couple which is proper over a noetherian henselian pair is henselian}
Let $(A,I)$ be an henselian pair with $A$ noetherian and let $X$ be a proper $A$-scheme. Set $S=Spec(A)$, $S_0=Spec(A/I)$ and let $X_0=X\times_S S_0$. Then $(X,X_0)$ is an henselian couple.
\end{lemma}
\begin{proof}
This is a trivial consequence of Theorem \ref{proper base change over henselian pairs} and \cite[Exposé XII, Proposition 6.5 (i)]{SGA4}.
\end{proof}
\begin{lemma}\label{if the pair associated to a couple is henselian, the couple is henselian}
Let $X$ be a scheme and let $X_0$ be a closed subscheme. Let $A$ be a noetherian ring and assume that $X$ is proper over $Spec(A)$. Also assume that $X_0=X\times_{Spec(A)} Spec(A/I)$ for some ideal $I\subseteq A$. Put $J=ker(B=\mathscr{O}_X(X)\longrightarrow \mathscr{O}_{X_0}(X_0))$. If $(B,J)$ is an henselian pair, then $(X,X_0)$ is an henselian couple.
\end{lemma}
\begin{proof}
Let $(A^h,I^h)$ be the henselization of the couple $(A,I)$ given by \cite[Tag 0A02]{SP}. Then we have the following diagram
\[
\begindc{\commdiag}[30]
\obj(50,20)[1]{$(X,X_0)$}
\obj(0,0)[2]{$(Spec(A^h),Spec(A^h/I^h))$}
\obj(50,0)[3]{$(Spec(A),Spec(A/I))$}
\mor{1}{3}{$f$}
\mor{2}{3}{$\gamma$}
\enddc
\]
which induces the following diagram of pairs:
\[
\begindc{\commdiag}[30]
\obj(30,20)[1]{$(B,J)$}
\obj(0,0)[2]{$(A^h,I^h)$}
\obj(30,0)[3]{$(A,I)$}
\mor{3}{1}{$$}
\mor{3}{2}{$$}
\mor{2}{1}{$\psi$}[1,1]
\enddc
\]
The morphism $\psi$ is the one induced by the universal property of $(A^h,I^h)$. As 
\[
Hom_{Rings}(A^h,B)=Hom_{Schemes}(X,Spec(A^h))
\]
the homomorphism $\psi$ identifies a unique morphism of schemes $\phi : X \longrightarrow Spec(A^h)$. Thus we get the following commutative diagram
\[
\begindc{\commdiag}[30]
\obj(30,20)[1]{$X$}
\obj(0,0)[2]{$Spec(A^h)$}
\obj(30,0)[3]{$Spec(A)$}
\mor{1}{2}{$\phi$}
\mor{2}{3}{$\gamma$}
\mor{1}{3}{$f$}
\enddc
\]
Moreover, by \cite[Tag 0AGU]{SP}, we get that
\[
\gamma^{-1}(Spec(A/I))=Spec(A^h\otimes_A A/I)=Spec(A^h/I^h)
\]
whence
\[
X\times_{Spec(A^h)} Spec(A^h/I^h)=X_0
\]
Therefore, the couple $(X,X_0)$ lies over the henselian couple $(Spec(A^h),Spec(A^h/I^h))$. Furthermore, $A^h$ is a noetherian ring (see \cite[Tag 0AGV]{SP}). Finally, as $f$ is a proper morphism and $\gamma$ is separated, we get that $\phi$ is proper as well by \cite[Proposition 3.3.16]{L}. Then we can conclude that $(X,X_0)$ is an henselian couple by the previous lemma. 
\end{proof}
The previous lemma tells us that, under some appropriate hypothesis, if the pair 
\[
(\mathscr{O}_X(X),ker(\mathscr{O}_X(X)\longrightarrow \mathscr{O}_{X_0}(X_0)))
\]
is henselian, then $(X,X_0)$ is an henselian couple. It is natural to ask if the converse is true, i.e. if given an henselian couple $(X,X_0)$ the associated pair is henselian. An answer is provided by the next lemma.
\begin{lemma}\label{henselian pair associated to an henselian couple}
Let $X$ be a quasi-compact and quasi-separated scheme and let $i: X_0\longrightarrow X$ be a closed immersion such that $(X,X_0)$ is an henselian couple.\\ Then $(B,J)=(\mathscr{O}_X(X),ker(\mathscr{O}_X(X)\longrightarrow \mathscr{O}_{X_0}(X_0)))$ is an henselian pair.
\end{lemma}
\begin{proof}
By \cite[Tag 09XI]{SP}, it is sufficient to show that for every étale ring map $B\longrightarrow C$ together with a $B$-morphism $\sigma: C\longrightarrow B/J$, there exists a $B$-morphism $C\longrightarrow B$ which lifts $\sigma$. \\
Consider the cartesian diagram
\[
\begindc{\commdiag}[20]
\obj(0,30)[1]{$X_C=X\times_{Spec(B)} Spec(C)$}
\obj(50,30)[2]{$X$}
\obj(0,0)[3]{$Spec(C)$}
\obj(50,0)[4]{$Spec(B)$}
\mor{1}{2}{$$}
\mor{1}{3}{$$}
\mor{2}{4}{$$}
\mor{3}{4}{$$}
\enddc
\]
As $Spec(C)\longrightarrow Spec(B)$ is étale and separated, the morphism $X_C \longrightarrow X$ is étale and separated as well. Then, by \cite[Proposition 18.5.4]{EGA4.4}, we have a bijection
\[
\Gamma(X_C/X)\longrightarrow \Gamma(X_C\times_X X_0/X_0)
\]
between the sections of $X_C\longrightarrow X$ and those of $X_C\times_X X_0\longrightarrow X_0$.\\
\textit{Observation 1.} The universal property of $X_C\times_X X_0$ tells us that
\[
\Gamma(X_C\times_X X_0/X_0)\cong Hom_X(X_0,X_C)
\]
\textit{Observation 2.} Let $\mathscr{J}\subseteq \mathscr{O}_X$ be the sheaf of ideals associated to $X_0$. Then we have a short exact sequence of $\mathscr{O}_X$-modules
\[
0\longrightarrow \mathscr{J} \longrightarrow \mathscr{O}_X \longrightarrow i_*\mathscr{O}_{X_0}\longrightarrow 0
\]
Applying the global sections functor, we get an exact sequence
\[
0\longrightarrow J=\mathscr{J}(X)\longrightarrow \mathscr{O}_X(X)=B \longrightarrow \mathscr{O}_{X_0}(X_0)
\]
Hence, we have an homomorphism
\[
B/J\longrightarrow \mathscr{O}_{X_0}(X_0)
\]
Therefore, we get a morphism of schemes 
\[
X_0 \longrightarrow Spec(\mathscr{O}_{X_0}(X_0))\longrightarrow Spec(B/J)
\]
Also notice that the diagram
\[
\begindc{\commdiag}[20]
\obj(0,30)[1]{$X_0$}
\obj(30,30)[2]{$X$}
\obj(0,0)[3]{$Spec(B/J)$}
\obj(30,0)[4]{$Spec(B)$}
\mor{1}{2}{$$}
\mor{1}{3}{$$}
\mor{2}{4}{$$}
\mor{3}{4}{$$}
\enddc
\]
is commutative.

Now consider the diagram
\[
\begindc{\commdiag}[20]
\obj(0,30)[1]{$X_C$}
\obj(30,30)[2]{$X$}
\obj(0,0)[3]{$Spec(C)$}
\obj(30,0)[4]{$Spec(B)$}
\obj(-30,50)[5]{$X_0$}
\obj(-30,-20)[6]{$Spec(B/J)$}
\mor{1}{2}{$$}
\mor{1}{3}{$$}
\mor{2}{4}{$$}
\mor{3}{4}{$$}
\mor{5}{2}{$$}
\mor{5}{3}{$$}
\mor{5}{6}{$$}
\mor(-28,48)(-2,32){$\exists ! \hspace{0.1cm}\tilde{\alpha}$}[0,1]
\mor{6}{3}{$$}
\mor{6}{4}{$$}
\enddc
\]
Label $\tilde{\alpha}: X_0\longrightarrow X_C$ the $X$-morphism provided by the universal property of $X_C$ and let $\alpha : X \longrightarrow X_C$ be the corresponding $X$-morphism in $\Gamma(X_C/X)$.\\
Consider the following commutative diagram
\[
\begindc{\commdiag}[20]
\obj(0,40)[1]{$X$}
\obj(30,40)[2]{$X_C$}
\obj(60,40)[3]{$X$}
\obj(30,20)[4]{$Spec(C)$}
\obj(60,20)[5]{$Spec(B)$}
\obj(60,0)[6]{$Spec(B/J)$}
\mor{1}{2}{$\alpha$}
\mor{2}{3}{$$}
\mor{2}{4}{$$}
\mor{3}{5}{$$}
\mor{4}{5}{$$}
\mor{6}{4}{$$}
\mor{6}{5}{$$}
\mor{1}{4}{$$}
\cmor((0,43)(2,48)(10,50)(30,50)(50,50)(58,48)(60,43)) \pdown (30,52){$id_X$}
\enddc
\]
and the corresponding commutative diagram in \textit{Rings}:
\[
\begindc{\commdiag}[20]
\obj(0,40)[1]{$B$}
\obj(30,40)[2]{$\mathscr{O}_{X_C}(X_C)$}
\obj(60,40)[3]{$B$}
\obj(30,20)[4]{$C$}
\obj(60,20)[5]{$B$}
\obj(60,0)[6]{$B/J$}
\mor{2}{1}{$\alpha$}
\mor{3}{2}{$$}
\mor{4}{2}{$$}
\mor{5}{3}{$id_B$}[\atright,\solidarrow]
\mor{5}{4}{$\phi$}
\mor{4}{6}{$\sigma$}[\atright,\solidarrow]
\mor{5}{6}{$\pi$}
\mor{4}{1}{$\psi$}
\cmor((0,43)(2,48)(10,50)(30,50)(50,50)(58,48)(60,43)) \pdown (30,52){$id_B$}
\enddc
\]
It is then clear that $\psi$ is the $B$-morphism we were looking for. This concludes the proof of the lemma.
\end{proof}
\begin{corollary}\label{an henselian couple proper over a noetherian ring is proper over an henselian pair}
Let $(X,X_0)$ be an henselian couple. Assume that $X$ is proper over a noetherian ring $A$ and that $X_0=X\times_{Spec(A)} Spec(A/I)$ for some ideal $I\subseteq A$. Then $(X,X_0)$ is proper over an henselian pair.
\end{corollary}
\begin{proof}
As $X$ is  proper over $Spec(A)$, it is a quasi-compact and quasi-separated scheme. Hence, by Lemma \ref{henselian pair associated to an henselian couple}, $(\mathscr{O}_X(X),ker(\mathscr{O}_X(X)\longrightarrow \mathscr{O}_{X_0}(X_0)))$ is an henselian pair. Therefore, by the same construction described in Lemma \ref{if the pair associated to a couple is henselian, the couple is henselian}, we get that $(X,X_0)$ is proper over $(A^h,I^h)$.
\end{proof}
\begin{corollary}\label{characterization of henselian couples proper over a noetherian ring}
Let $(X,X_0)$ be a couple and assume that $X$ is proper over a noetherian ring $A$ and that $X_0=X\times_{Spec(A)} Spec(A/I)$ for some ideal $I\subseteq A$. Then $(X,X_0)$ is an henselian couple if and only if $(\mathscr{O}_X(X),ker(\mathscr{O}_X(X)\longrightarrow \mathscr{O}_{X_0}(X_0)))$ is an henselian pair.
\end{corollary}
  


By Remark \ref{first step in the solution} every henselian couple $(X,X_0)$ which arises as in Lemma \ref{a couple which is proper over a noetherian henselian pair is henselian} satisfies conditions $\textit{2.}$ and $\textit{3.}$ in \cite[Exposé XII, Proposition 6.5]{SGA4} with $\mathbb{L}=\mathbb{P}$ and for every $n$. 
Then, applying Corollary \ref{an henselian couple proper over a noetherian ring is proper over an henselian pair}, we get the following result:
\begin{theorem}
Let $(X,X_0)$ be an henselian couple. Assume that $X$ is proper over a noetherian ring $A$ and that $X_0=X\times_{Spec(A)} Spec(A/I)$ for some ideal $I\subseteq A$. Then conditions $\textit{2.}$ and $\textit{3.}$ in \textnormal{\cite[Exp. XII, Remarks 6.13]{SGA4}} are satisfied with $\mathbb{L}=\mathbb{P}$ and for every $n$.
\end{theorem}

This gives a positive answer to Question \ref{conj1} if we assume that hypothesis $(\dagger)$ hold.

\textbf{Acknowledgments.} A special thank you to Moritz Kerz. It is worthy to mention that he introduced me to the problem treated in this paper. In particular, I would like to point out that he mentioned Popescu's Theorem to me, which I did not know until then, grasping the fact that it could have been an helpful tool for my purposes. I also wish to thank him for the time he dedicated to the review of this paper.\\
I also wish to thank Federico Binda for the many interesting discussions I had with him and for his precious advices.

\end{document}